\documentclass[11pt]{amsart}

\usepackage[margin=1.15in]{geometry}
\usepackage{amsmath,amssymb,amsthm,mathtools}
\usepackage{mathrsfs}
\usepackage{enumitem}
\usepackage{xcolor}
\usepackage[colorlinks=true,citecolor=blue,linkcolor=blue,urlcolor=blue]{hyperref}
\usepackage[nameinlink,noabbrev]{cleveref}
\usepackage{mathabx}

\theoremstyle{plain}
\newtheorem{theorem}{Theorem}[section]
\newtheorem{proposition}[theorem]{Proposition}
\newtheorem{lemma}[theorem]{Lemma}
\newtheorem{corollary}[theorem]{Corollary}
\newtheorem{fact}[theorem]{Fact}

\theoremstyle{definition}
\newtheorem{definition}[theorem]{Definition}

\newtheorem{example}[theorem]{Example}

\theoremstyle{remark}
\newtheorem{remark}[theorem]{Remark}

\newcommand{\cl}{\operatorname{cl}}
\newcommand{\ci}{\operatorname{ci}}
\newcommand{\cf}{\operatorname{cf}}
\newcommand{\dcl}{\operatorname{dcl}}
\newcommand{\st}{\operatorname{st}}

\newcommand{\tame}{\operatorname{tame}}

\title{Definability of Hausdorff Limits for Lipschitz Cells in O-minimal Structures}

\author{Xiaoduo Wang}
\address{Department of Mathematics, University of Manchester}
\email{xiaoduo.wang@manchester.ac.uk}

\date{\today}

\begin{document}

\begin{abstract}
We study Hausdorff limits of definable families over arbitrary models of o-minimal expansions of real closed fields. Over the real field, van den Dries proved that Hausdorff limits of definable families are definable, giving a geometric interpretation of the Marker--Steinhorn theorem. We prove a non-Archimedean analogue for definable families which are Lipschitz cells with a fixed cell presentation and a uniform Lipschitz bound. The proof replaces compactness of closed and bounded subsets of \(\mathbb R^n\) by dense completions and long Cauchy sequences, and treats the Hausdorff distance as a metric valued in an ordered completion. We show that Hausdorff limits of such families are standard parts of external fibers over tame extensions, and use stable embeddedness of tame pairs to prove that every such limit is definable in the dense completion of the base model. We also prove a uniform version: the collection of these Hausdorff limits forms a definable family in the dense completion.
\end{abstract}

\maketitle

\section{Introduction}

O-minimality was introduced by Pillay and Steinhorn as a model-theoretic framework for tame geometry in ordered structures \cite{PillaySteinhorn1986}. An expansion of a dense linear order without endpoints is called \emph{o-minimal} if every definable subset of the universe is a finite union of points and intervals. Although this is a one-dimensional condition, it has strong geometric consequences in all dimensions, including cell decomposition, monotonicity, dimension theory, and definable choice. We refer to van den Dries \cite{vandenDries2003} for the general background.

The starting point of this paper is the Marker--Steinhorn theorem. Recall that a type \(p(x)\) over a model \(M\) is \emph{definable} if, for every formula \(\varphi(x,y)\), the set
\[
\{b\in M^{|y|}:\varphi(x,b)\in p\}
\]
is definable over \(M\). Marker and Steinhorn proved in \cite{MarkerSteinhorn1994} that, in o-minimal theories, definability of types is controlled by an order-theoretic tameness condition on elementary extensions. In the general o-minimal setting this condition is phrased in terms of Dedekind completeness inside the elementary extension generated by a realization of the type. A gap in the original proof was later identified and repaired by Andujar Guerrero \cite{AndujarGuerrero2025}.

When the theory expands the theory of real closed fields, the tameness condition in the Marker--Steinhorn theorem can be expressed using standard parts. This leads to tame pairs: one considers an elementary extension together with a predicate for the smaller model and the associated standard part map. The key feature used here is stable embeddedness of the small model in the pair. In particular, subsets of the small model definable externally are already definable in the original structure. This is the model-theoretic mechanism which later turns external representatives of Hausdorff limits into definable sets.

The literature around the Marker--Steinhorn theorem has several complementary directions. Pillay studied definability of types and pairs of o-minimal structures in \cite{Pillay1994}. Van den Dries and Lewenberg developed the theory of \(T\)-convexity and tame extensions in \cite{vandendriesLewenberg1995,vandenDries1997}. Tressl gave a valuation-theoretic account of the Marker--Steinhorn theorem in \cite{Tressl2004}, Walsberg gave a proof using definable linear orders in \cite{Walsberg2019}, and the correction mentioned above is due to Andujar Guerrero \cite{AndujarGuerrero2025}. Related work on externally definable sets and dependent pairs includes Chernikov and Simon \cite{ChernikovSimon2012}.

Van den Dries later gave a geometric interpretation of the Marker--Steinhorn theorem in terms of Hausdorff limits of definable families over the real field \cite{Lisbon2003}. Let \(A\subseteq R^{m+n}\) be a definable family, and write \(A_a\subseteq R^n\) for the fiber over \(a\in R^m\). If a sequence of fibers \(A_{a_i}\) converges in the Hausdorff metric, one can ask whether the limiting set is definable, and whether all such limits form a definable family. Over the real field, the answer is positive. From the model-theoretic point of view, a Hausdorff limit is represented as the standard part of an external fiber, and stable embeddedness of tame pairs gives definability. Lion and Speissegger gave another proof of the same definability theorem in \cite{SpeisseggerLion2004}, using a geometric construction of representations based on tangent spaces, Grassmannians, blow-ups, and integrable distributions.

The aim of this paper is to extend part of this geometric interpretation from the real field to arbitrary models of o-minimal expansions of real closed fields. Over \(\mathbb R\), closed and bounded definable sets are compact, and ordinary sequences are enough to detect Hausdorff limits. Over a non-Archimedean real closed field, neither feature remains true in the same form. We replace compactness by dense completions and long Cauchy sequences, and we replace ordinary real-valued Hausdorff distance by a Hausdorff distance taking values in an ordered completion. This is why cuts, completions, and monoid-valued metrics enter the argument.

Conceptually, the proof follows the standard-part pattern familiar from tame pairs. A limiting process is first encoded by an object in an elementary extension; the relevant bounded part is pushed back to the base by the standard part map; stable embeddedness then guarantees that no new definable subsets of the base have been introduced. The new difficulty is to make this pattern work for Hausdorff convergence over non-Archimedean models, where one has to replace ordinary sequential compactness by long Cauchy sequences and dense completion.

Throughout the paper, \(T\) is a complete o-minimal expansion of the theory of real closed fields in a language \(\mathcal L\), and \(\mathcal M\models T\). Unless otherwise stated, definable means \(\mathcal L(M)\)-definable. We use standard facts from o-minimality without further comment, including cell decomposition, monotonicity, definable choice, dimension theory, and the o-minimal curve selection lemma; see \cite{vandenDries2003}. All metric notions on powers of \(M\) are taken with respect to the sup metric unless otherwise specified.

The main contribution of this paper is the Lipschitz-cell case of the non-Archimedean Hausdorff-limit theorem. First, we develop the metric background needed to discuss Hausdorff limits over arbitrary real closed fields. This includes a completion construction for \(R\)-metric spaces using long Cauchy sequences, and a Hausdorff distance taking values in the ordered completion of the base field.

Second, we prove that if \(A\subseteq M^{m+n}\) is a bounded \(\mathcal C^p\) Lipschitz cell whose fibers have a fixed cell presentation with a uniform Lipschitz bound, then Hausdorff-Cauchy sequences of fibers have Lipschitz limit cells. This is the geometric preparation which replaces compactness of closed and bounded subsets of \(\mathbb R^n\).

Third, we show that Hausdorff limits of such fibers are standard parts of external fibers over tame extensions. We then remove the non-tame parameters from these external representatives, and use stable embeddedness of tame pairs to prove that every Hausdorff limit is definable in the dense completion of the base model; see Theorem~\ref{Hausdorff limits of Lipschitz cells are definable in dense completion}.

Finally, we prove a uniform version: the collection of all Hausdorff limits of fibers from such a family is itself a definable family in the dense completion; see Theorem~\ref{Hausdorff limits form definable family in dense completion}.

The hypothesis that the whole family is a Lipschitz cell is the point at which the present paper stops short of the full real-field theorem. A finite cell decomposition of a Hausdorff-Cauchy definable family need not itself produce Hausdorff-Cauchy families on the individual cells. Removing this hypothesis requires an additional decomposition or gluing argument. Thus the present paper should be viewed as the Lipschitz-cell case of the non-Archimedean analogue of the definability theorem for Hausdorff limits. The full analogue, as well as a corresponding extension of the Lion--Speissegger geometric proof, is left for future work.

The paper is organized as follows. We first recall the metric background needed over non-Archimedean real closed fields: cuts, completions, monoid-valued metrics, long sequences, and Hausdorff distance. We then prove the existence of Hausdorff limits for Cauchy sequences of fibers of Lipschitz cell families and represent these limits as standard parts of external fibers. Finally, we remove non-tame parameters from the external fibers and apply stable embeddedness of tame pairs to prove pointwise and uniform definability.

\section{Preliminaries}

We keep the conventions from the introduction. Thus \(T\) is a complete o-minimal expansion of the theory of real closed fields in a language \(\mathcal L\), \(\mathcal M\models T\), and definable means \(\mathcal L(M)\)-definable unless otherwise stated. We collect the model-theoretic and order-theoretic tools used later.

\subsection{Tame extensions and standard parts}

Let \(\mathcal M\preceq\mathcal N\models T\). An element \(a\in N\) is \emph{\(M\)-bounded} if there is \(b\in M^{>0}\) such that \(|a|<b\). It is \emph{\(M\)-infinitesimal} if \(|a|<b\) for every \(b\in M^{>0}\). It is \emph{infinite with respect to \(M\)} if it is not \(M\)-bounded.

The extension \(\mathcal N\) of \(\mathcal M\) is called \emph{tame} if every \(M\)-bounded element of \(N\) is infinitesimally close to an element of \(M\). Equivalently, for every \(M\)-bounded \(a\in N\), there is a unique \(a_0\in M\) such that \(a-a_0\) is \(M\)-infinitesimal. We write \(\st_M(a):=a_0\), and call it the \emph{standard part} of \(a\) over \(M\). We apply \(\st_M\) coordinatewise to tuples.

We use the language \(\mathcal L_{\mathrm{tame}}:=\mathcal L\cup\{U,\st\}\), where \(U\) is a unary predicate and \(\st\) is a unary function symbol. If \(\mathcal N\) is tame over \(\mathcal M\), then we regard \((\mathcal N,\mathcal M,\st_M)\) as an \(\mathcal L_{\mathrm{tame}}\)-structure by interpreting \(U\) as \(M\), interpreting \(\st\) as the standard part map on \(M\)-bounded elements, and assigning an arbitrary value, say \(0\), to elements infinite with respect to \(M\).

Let \(\mathcal M\preceq\mathcal N\). A set \(X\subseteq M^n\) is \emph{externally definable} if there is an \(\mathcal L(N)\)-definable set \(Y\subseteq N^n\) such that \(X=Y\cap M^n\). We say that \(\mathcal M\) is \emph{stably embedded} in \(\mathcal N\) if every externally definable subset of \(M^n\), for every \(n\), is definable in \(\mathcal M\) with parameters from \(M\).

\begin{fact}[Stable embeddedness of tame pairs]\label{Stable embeddedness of tame pairs}
Let \((\mathcal N,\mathcal M,\st_M)\) be a tame pair of models of \(T\). Then \(\mathcal M\) is stably embedded in the tame-pair structure, and the structure induced on \(M\) is precisely its original \(\mathcal L\)-structure.
\end{fact}

This follows from the theory of tame pairs developed by van den Dries--Lewenberg \cite{vandendriesLewenberg1995,vandenDries1997}; see also \cite[Proposition~8.1]{Lisbon2003}. In particular, if \(X\subseteq M^n\) is definable in a tame-pair structure with parameters from the large model, then \(X\) is already \(\mathcal L(M)\)-definable.

\subsection{Cuts and monoid-valued metrics}

Throughout this subsection, \((G,+,<)\) is a divisible ordered abelian group. We identify a cut with its lower part. Thus a cut of \(G\) is a downward closed subset of the underlying ordered set \((G,<)\). When convenient, we also write a cut as a pair \(\xi=(\xi^L,\xi^R)\), where \(\xi^L\) is the lower part and \(\xi^R=G\setminus\xi^L\). The order on cuts is inclusion of lower parts.

We write \(\check G\) for the set of cuts of \(G\). The improper cuts are denoted by \(-\infty=(\varnothing,G)\) and \(+\infty=(G,\varnothing)\). For \(g\in G\), the two principal cuts determined by \(g\) are denoted by
\[
g^-:=((-\infty,g),[g,\infty)),\qquad
g^+:=((-\infty,g],(g,\infty)).
\]
A cut is \emph{principal} if it is one of these cuts, and is \emph{non-principal} otherwise.

For \(g\in G\) and \(\xi=(\xi^L,\xi^R)\in\check G\), set \(g+\xi:=(g+\xi^L,g+\xi^R)\). The \emph{invariance group} of \(\xi\) is \(G(\xi):=\{g\in G:g+\xi=\xi\}\). A cut \(\xi\) is called \emph{dense} if it is non-principal and \(G(\xi)=\{0\}\).

For cuts \(\xi,\eta\in\check G\), define the \emph{left sum} and \emph{right sum} by
\[
\xi+\eta:=(\xi^L+\eta^L,G\setminus(\xi^L+\eta^L)),
\qquad
\xi+^R\eta:=(G\setminus(\xi^R+\eta^R),\xi^R+\eta^R).
\]

\begin{proposition}[\cite{FornasieroMamino2008}]\label{Cut sums are ordered monoids}
The structures \((\check G,+,0^+,<)\) and \((\check G,+^R,0^-,<)\) are ordered commutative monoids.
\end{proposition}

Define an equivalence relation \(\sim\) on \(\check G\) by identifying \(g^-\sim g^+\) for every \(g\in G\), and making no other identifications. Let \(\overline G:=\check G/\sim\). We identify the common class of \(g^-\) and \(g^+\) with \(g\), and regard \(G\) as a subset of \(\overline G\). The operations \(+\) and \(+^R\) descend to \(\overline G\).

\begin{corollary}\label{Positive completion is ordered monoid}
The structures \((\overline G^{\geq0},+,0,<)\) and \((\overline G^{\geq0},+^R,0,<)\) are ordered commutative monoids.
\end{corollary}

\begin{lemma}\label{Dense in Nonprincipal Cuts}
For every positive element \(\xi\in\overline G\), there is \(g\in G^{>0}\) such that \(0<g<\xi\). In particular, \(G^{>0}\) is coinitial in \(\overline G^{>0}\).
\end{lemma}

\begin{proof}
If \(\xi\in G\), this is immediate. Suppose that \(\xi\) is represented by a positive non-principal cut. Then \(0\in\xi^L\), and since \(\xi^L\) has no greatest element, there is \(g\in\xi^L\) with \(g>0\). Hence \(0<g<\xi\).
\end{proof}

\begin{lemma}\label{Supremum and Infimum of Subsets of Ordered Group Exist in Completion}
Let \(A\subseteq G\). Then \(\inf A\) and \(\sup A\) exist in \(\overline G\).
\end{lemma}

\begin{proof}
If \(\inf A\) exists in \(G\), then it is also the infimum of \(A\) in \(\overline G\). Otherwise, set \(A_+:=\{g\in G:\exists a\in A,\ a<g\}\). Then \((G\setminus A_+,A_+)\) is the cut determined by the lower bounds of \(A\), and is the infimum of \(A\) in \(\overline G\). The proof for \(\sup A\) is dual.
\end{proof}

\begin{corollary}\label{Supremum and Infimum of Subsets of Completion Exist}
For every \(A\subseteq\overline G\), the elements \(\inf A\) and \(\sup A\) exist in \(\overline G\).
\end{corollary}

\begin{proof}
We prove the statement for infima; suprema are analogous. Let \(A_0:=A\cap G\) and \(A_1:=A\setminus A_0\). By Lemma~\ref{Supremum and Infimum of Subsets of Ordered Group Exist in Completion}, \(\inf A_0\) exists in \(\overline G\). For \(A_1\), define \(\xi^R:=\bigcup_{\zeta\in A_1}\zeta^R\) and \(\xi:=(G\setminus\xi^R,\xi^R)\). Then \(\xi=\inf A_1\). Hence \(\inf A=\min\{\inf A_0,\inf A_1\}\).
\end{proof}

\begin{definition}\label{Definition of monoid valued metric}
Let \((\mathcal G,+,0,<)\) be an ordered commutative monoid with least element \(0\). A function \(d:Z\times Z\to\mathcal G\) is called a \emph{\(\mathcal G\)-pseudometric} on \(Z\) if \(d(x,x)=0\), \(d(x,y)=d(y,x)\), and \(d(x,z)\leq d(x,y)+d(y,z)\) for all \(x,y,z\in Z\). If moreover \(d(x,y)=0\) implies \(x=y\), then \(d\) is called a \emph{\(\mathcal G\)-metric}. When \(G\) is an ordered abelian group, a \(G\)-pseudometric or \(G\)-metric means a \(G^{\geq0}\)-valued pseudometric or metric.
\end{definition}

\begin{definition}\label{Definition of balls and neighborhoods}
Let \((Z,\mathcal G,d)\) be a \(\mathcal G\)-pseudometric space. For \(x\in Z\) and \(\varepsilon\in\mathcal G\) with \(\varepsilon>0\), define \(B_d(x,\varepsilon):=\{y\in Z:d(x,y)<\varepsilon\}\). For \(A\subseteq Z\), define \(U_d(A,\varepsilon):=\bigcup_{x\in A}B_d(x,\varepsilon)\). When \(d\) is clear from context, we omit it from the notation.
\end{definition}

\begin{proposition}[Tressl, Proposition 3.1 in \cite{Tressl2025}]\label{Dense cuts and unique realizations}
Let \(K\) be an ordered field and let \(\xi\) be a cut of \(K\). Then \(\xi\) is dense if and only if there is an ordered field extension \(L\) of \(K\) in which \(\xi\) has a unique realization.
\end{proposition}

\begin{proposition}[Tressl, Corollary 3.3 in \cite{Tressl2006}]\label{Existence of dense completion}
Let \(T\) be an o-minimal expansion of the theory of real closed fields, and let \(R\models T\). Then there is a model \(S\succ R\) such that \(R\) is dense in \(S\), and if \(R'\succ R\) with \(R\) dense in \(R'\), then there is a unique elementary embedding \(R'\to S\) over \(R\). Moreover, \(S\) is uniquely determined up to a unique \(R\)-isomorphism by these properties.
\end{proposition}

\begin{definition}\label{Definition of dense completion}
The model \(S\) in Proposition~\ref{Existence of dense completion} is called the \emph{dense completion} of \(R\). If \(S=R\), then \(R\) is called \emph{dense complete}.
\end{definition}

\section{Hausdorff Limits in \(R\)-Metric Spaces}

In this section we develop the metric formalism used later to study Hausdorff limits over non-Archimedean real closed fields. The main difference from the classical real case is that sequences must be indexed by a cofinality adapted to the value group of the metric, and Cauchy limits may naturally live in the dense completion of the base field. We first recall long sequences in \(R\)-metric spaces and the corresponding completion construction. We then define the Hausdorff distance for subsets of an \(R\)-metric space and record the basic properties of Hausdorff convergence needed in the sequel.

\subsection{Long sequences and metric completions}

Throughout this subsection, \(R\) is a real closed field, \((Z,R,d)\) is an \(R\)-metric space, and
\[
\kappa:=\ci(R^{>0}).
\]
Thus \(\kappa\) is the least cardinality of a coinitial subset of \(R^{>0}\). The balls \(B_d(x,\varepsilon)\), where \(x\in Z\) and \(\varepsilon\in R^{>0}\), form a basis for a topology on \(Z\). We call this topology the \emph{\(R\)-metric topology} induced by \(d\).

\begin{lemma}\label{R-metric is continuous}
The distance function \(d:Z\times Z\to R\) is continuous, where \(R\) is equipped with the order topology and \(Z\times Z\) is equipped with the product topology induced by the \(R\)-metric topology.
\end{lemma}

\begin{proof}
Let \((a,b)\subseteq R\) be an open interval, and fix \((x,y)\in d^{-1}((a,b))\). Set \(c=d(x,y)\), and choose
\[
m=\min\left\{\frac{c-a}{2},\frac{b-c}{2}\right\}>0.
\]
If \(x'\in B_d(x,m)\) and \(y'\in B_d(y,m)\), then \(d(x',y')\leq d(x',x)+d(x,y)+d(y,y')<2m+c<b\), and \(d(x',y')\geq d(x,y)-d(x,x')-d(y,y')>c-2m>a\). Thus \(B_d(x,m)\times B_d(y,m)\subseteq d^{-1}((a,b))\), so \(d^{-1}((a,b))\) is open.
\end{proof}

\begin{definition}\label{Definition of long sequence}
Let \(\lambda\) be a limit ordinal. A \emph{\(\lambda\)-sequence} in \(Z\) is a family \((x_\alpha)_{\alpha<\lambda}\). If \(\mu\) is another limit ordinal, a \emph{\(\mu\)-subsequence} of \((x_\alpha)_{\alpha<\lambda}\) is a sequence of the form \((x_{g(\beta)})_{\beta<\mu}\), where \(g:\mu\to\lambda\) is strictly increasing and cofinal in \(\lambda\).
\end{definition}

\begin{definition}\label{Definition of convergence and Cauchy long sequence}
Let \(\lambda\) be a limit ordinal. A \(\lambda\)-sequence \((x_\alpha)_{\alpha<\lambda}\) in \(Z\) converges to \(x\in Z\), written \(x_\alpha\to x\) as \(\alpha\to\lambda\), if for every \(\varepsilon\in R^{>0}\), there is \(\gamma<\lambda\) such that \(d(x_\alpha,x)<\varepsilon\) for all \(\gamma<\alpha<\lambda\). The sequence is called \emph{Cauchy} if for every \(\varepsilon\in R^{>0}\), there is \(\gamma<\lambda\) such that \(d(x_\alpha,x_\beta)<\varepsilon\) for all \(\gamma<\alpha,\beta<\lambda\).
\end{definition}

\begin{lemma}\label{Nontrivial Cauchy sequences have cofinality kappa}
Let \(\lambda\) be a limit ordinal, and let \((x_\alpha)_{\alpha<\lambda}\) be a Cauchy \(\lambda\)-sequence in \(Z\). If \(\cf(\lambda)\neq\kappa\), then \((x_\alpha)_{\alpha<\lambda}\) is eventually constant.
\end{lemma}

\begin{proof}
First suppose that \(\cf(\lambda)<\kappa\). Let \(I\subseteq\lambda\) be cofinal with \(|I|<\kappa\). If the sequence is not eventually constant, then for every \(\gamma\in I\) there are \(\alpha_\gamma,\beta_\gamma>\gamma\) such that \(x_{\alpha_\gamma}\neq x_{\beta_\gamma}\). Put \(r_\gamma=d(x_{\alpha_\gamma},x_{\beta_\gamma})>0\). Since \(\{r_\gamma:\gamma\in I\}\) has cardinality \(<\kappa\), it is not coinitial in \(R^{>0}\). Choose \(\varepsilon\in R^{>0}\) such that \(\varepsilon<r_\gamma\) for every \(\gamma\in I\). By Cauchyness, there is \(\gamma_0<\lambda\) such that \(d(x_\alpha,x_\beta)<\varepsilon\) for all \(\gamma_0<\alpha,\beta<\lambda\). Choose \(\gamma\in I\) with \(\gamma>\gamma_0\). Then \(r_\gamma<\varepsilon\), a contradiction.

Now suppose that \(\cf(\lambda)>\kappa\). Choose a coinitial family \((\varepsilon_i)_{i<\kappa}\) in \(R^{>0}\). For each \(i<\kappa\), choose \(\gamma_i<\lambda\) such that \(d(x_\alpha,x_\beta)<\varepsilon_i\) for all \(\gamma_i<\alpha,\beta<\lambda\). Since \(\cf(\lambda)>\kappa\), the ordinal \(\gamma^*:=\sup_{i<\kappa}\gamma_i\) is still below \(\lambda\). Hence, for all \(\gamma^*<\alpha,\beta<\lambda\), we have \(d(x_\alpha,x_\beta)<\varepsilon_i\) for every \(i<\kappa\). Since \((\varepsilon_i)_{i<\kappa}\) is coinitial in \(R^{>0}\), this implies \(d(x_\alpha,x_\beta)=0\). Thus the sequence is eventually constant.
\end{proof}

\begin{definition}\label{Definition of complete R metric space}
The \(R\)-metric space \((Z,R,d)\) is called \emph{complete} if every Cauchy \(\kappa\)-sequence in \(Z\) converges.
\end{definition}

Lemma~\ref{Nontrivial Cauchy sequences have cofinality kappa} shows that completeness is naturally tested on Cauchy \(\kappa\)-sequences. Indeed, every non-eventually constant Cauchy sequence has cofinality \(\kappa\), and passing to cofinal \(\kappa\)-subsequences preserves convergence and Cauchyness.

\begin{lemma}\label{Long Sequence and Closure}
Let \(A\subseteq Z\). Then \(x\in\cl_d(A)\) if and only if there is a \(\kappa\)-sequence \((x_\alpha)_{\alpha<\kappa}\) in \(A\) such that \(x_\alpha\to x\).
\end{lemma}

\begin{proof}
Suppose first that \(x\in\cl_d(A)\). Choose a decreasing coinitial sequence \((\varepsilon_\alpha)_{\alpha<\kappa}\) in \(R^{>0}\). For each \(\alpha<\kappa\), choose \(x_\alpha\in A\cap B_d(x,\varepsilon_\alpha)\). Then \(x_\alpha\to x\). Conversely, if \(x_\alpha\to x\) with \(x_\alpha\in A\), then every ball around \(x\) intersects \(A\). Hence \(x\in\cl_d(A)\).
\end{proof}

\begin{lemma}\label{Long Sequences and Continuity}
Let \((X,R,d)\) and \((X',R,d')\) be \(R\)-metric spaces, let \(f:X\to X'\), and let \(x\in X\). Then \(f\) is continuous at \(x\) if and only if for every \(\kappa\)-sequence \((x_\alpha)_{\alpha<\kappa}\) in \(X\) with \(x_\alpha\to x\), we have \(f(x_\alpha)\to f(x)\).
\end{lemma}

\begin{proof}
Suppose first that \(f\) is continuous at \(x\), and let \(x_\alpha\to x\). For every neighbourhood \(V\) of \(f(x)\), choose a neighbourhood \(U\) of \(x\) such that \(f(U)\subseteq V\). Since \(x_\alpha\to x\), eventually \(x_\alpha\in U\), and hence eventually \(f(x_\alpha)\in V\).

Conversely, suppose that \(f\) is not continuous at \(x\). Then there is a neighbourhood \(V\) of \(f(x)\) such that every neighbourhood of \(x\) contains some \(y\) with \(f(y)\notin V\). Choose a decreasing coinitial sequence \((\varepsilon_\alpha)_{\alpha<\kappa}\) in \(R^{>0}\). For each \(\alpha<\kappa\), choose \(x_\alpha\in B_d(x,\varepsilon_\alpha)\) with \(f(x_\alpha)\notin V\). Then \(x_\alpha\to x\), but \(f(x_\alpha)\) does not converge to \(f(x)\).
\end{proof}

\begin{definition}\label{Definition of Lipschitz and uniformly continuous maps}
Let \((X,R,d)\) and \((X',R,d')\) be \(R\)-metric spaces, and let \(\lambda\in R^{>0}\). A function \(f:X\to X'\) is called \emph{\(\lambda\)-Lipschitz} if \(d'(f(x),f(y))\leq\lambda d(x,y)\) for all \(x,y\in X\). The function \(f\) is called \emph{uniformly continuous} if for every \(\varepsilon\in R^{>0}\), there is \(\delta\in R^{>0}\) such that \(d(x,y)<\delta\) implies \(d'(f(x),f(y))<\varepsilon\).
\end{definition}

\begin{lemma}\label{Lipschitz Implies Continuity}
Let \(f:(X,R,d)\to(X',R,d')\) be \(\lambda\)-Lipschitz. Then \(f\) is continuous.
\end{lemma}

\begin{proof}
If \(x_\alpha\to x\), then \(d'(f(x_\alpha),f(x))\leq\lambda d(x_\alpha,x)\), so \(f(x_\alpha)\to f(x)\). The result follows from Lemma~\ref{Long Sequences and Continuity}.
\end{proof}

\begin{definition}\label{Definition of isometry}
Let \((X,R,d)\) and \((X',R,d')\) be \(R\)-metric spaces. A map \(f:X\to X'\) is called an \emph{isometry} if \(d'(f(x),f(y))=d(x,y)\) for all \(x,y\in X\). It is called an \emph{isometric embedding} if it is an isometry onto its image.
\end{definition}

We now begin the completion construction for an \(R\)-metric space after passing from \(R\) to its dense completion. Let \((X,R,d)\) be an \(R\)-metric space, and let \(S\) be the dense completion of \(R\). Since \(R\) is dense in \(S\), the same cardinal \(\kappa=\ci(R^{>0})=\ci(S^{>0})\) is sufficient for testing convergence and Cauchyness after distances are regarded as \(S\)-valued.

Let
\[
\mathcal C_\kappa(X):=
\{(a_i)_{i<\kappa}:(a_i)_{i<\kappa}\text{ is a Cauchy }\kappa\text{-sequence in }X\}.
\]
Define an equivalence relation \(\sim\) on \(\mathcal C_\kappa(X)\) by declaring \((a_i)_{i<\kappa}\sim(b_i)_{i<\kappa}\) if, for every \(\varepsilon\in R^{>0}\), there is \(\alpha<\kappa\) such that \(d(a_i,b_i)<\varepsilon\) for all \(\alpha<i<\kappa\). Let
\[
\overline X:=\mathcal C_\kappa(X)/\sim.
\]

\begin{lemma}\label{Dense Complete implies Cauchy Complete}
Let \(R\) be dense complete. Then the \(R\)-metric space \((R,R,d)\), where \(d(x,y)=|x-y|\), is complete.
\end{lemma}

\begin{proof}
Let \((a_i)_{i<\kappa}\) be a Cauchy \(\kappa\)-sequence in \(R\). For each \(\alpha<\kappa\), define
\[
L_\alpha:=\inf_{\alpha<i<\kappa}a_i,
\qquad
U_\alpha:=\sup_{\alpha<i<\kappa}a_i
\]
in the order completion \(\overline R\). Set \(\xi_-:=\sup_{\alpha<\kappa}L_\alpha\) and \(\xi_+:=\inf_{\alpha<\kappa}U_\alpha\). We first show that \(\xi_-=\xi_+\). Suppose not. Choose \(g,g'\in R\) with \(\xi_-<g<g'<\xi_+\), and set \(\varepsilon:=g'-g>0\). Since \((a_i)_{i<\kappa}\) is Cauchy, choose \(\alpha<\kappa\) such that \(|a_i-a_j|<\varepsilon\) for all \(\alpha<i,j<\kappa\). But \(L_\alpha<g\) and \(g'<U_\alpha\), so there are \(i,j>\alpha\) such that \(a_i<g\) and \(g'<a_j\), a contradiction.

Let \(\xi\) be the common value of \(\xi_-\) and \(\xi_+\). We claim that \(\xi\) is either principal or dense. Suppose that \(\xi\) is non-principal. Let \(r\in R^{>0}\). Choose \(\alpha<\kappa\) such that \(|a_i-a_j|<r/3\) for all \(\alpha<i,j<\kappa\), and fix \(i>\alpha\). Then \(a_i-r/3\leq L_\alpha\leq \xi\leq U_\alpha\leq a_i+r/3\). Put \(q:=a_i-r/2\). Then \(q<\xi\), while \(q+r>\xi\). Hence translation by \(r\) cannot fix \(\xi\). Since this holds for every \(r>0\), the invariance group of \(\xi\) is trivial, and therefore \(\xi\) is dense.

Since \(R\) is dense complete, the cut \(\xi\) is realized by some \(a\in R\). Let \(\varepsilon\in R^{>0}\). Choose \(\alpha<\kappa\) such that \(|a_i-a_j|<\varepsilon\) for all \(\alpha<i,j<\kappa\). Since \(\xi\) is realized by \(a\), every sufficiently late \(a_i\) lies in \((a-\varepsilon,a+\varepsilon)\). Hence \(a_i\to a\).
\end{proof}

For \([(a_i)],[(b_i)]\in\overline X\), define
\[
\overline d([(a_i)],[(b_i)])
:=
\lim_{i\to\kappa}d(a_i,b_i),
\]
where the limit is taken in \(S\).

\begin{proposition}\label{Canonical Completion of Metric Space}
The function \(\overline d\) is a well-defined \(S\)-metric on \(\overline X\), and \((\overline X,S,\overline d)\) is complete. Moreover, the map \(\iota:X\to\overline X\), \(x\mapsto[(x)_{i<\kappa}]\), is an isometric embedding, and \(\iota(X)\) is dense in \(\overline X\).
\end{proposition}

\begin{proof}
Let \((a_i)\) and \((b_i)\) be Cauchy \(\kappa\)-sequences in \(X\). Then \((d(a_i,b_i))_{i<\kappa}\) is Cauchy in \(R\), since
\[
|d(a_i,b_i)-d(a_j,b_j)|
\leq d(a_i,a_j)+d(b_i,b_j).
\]
As \(S\) is dense complete, Lemma~\ref{Dense Complete implies Cauchy Complete} gives a limit in \(S\). This defines \(\overline d\). If \((a_i)\sim(a_i')\) and \((b_i)\sim(b_i')\), then the same inequality shows that \(\lim_i d(a_i,b_i)=\lim_i d(a_i',b_i')\), so \(\overline d\) is well-defined.

The metric axioms follow by passing to limits from the corresponding metric axioms for \(d\). The map \(\iota\) is distance-preserving by definition. To see that \(\iota(X)\) is dense, let \([(a_i)]\in\overline X\) and let \(\varepsilon\in S^{>0}\). Choose \(r\in R^{>0}\) with \(0<r<\varepsilon\). Since \((a_i)\) is Cauchy, for all sufficiently large \(i<\kappa\), we have \(\overline d([(a_i)],\iota(a_i))<r<\varepsilon\).

It remains to prove completeness. Let \((z_\alpha)_{\alpha<\kappa}\) be a Cauchy \(\kappa\)-sequence in \(\overline X\), and choose a decreasing coinitial sequence \((\varepsilon_\alpha)_{\alpha<\kappa}\) in \(R^{>0}\). Since \(\iota(X)\) is dense in \(\overline X\), choose \(x_\alpha\in X\) such that \(\overline d(z_\alpha,\iota(x_\alpha))<\varepsilon_\alpha\). Then \((x_\alpha)_{\alpha<\kappa}\) is Cauchy in \(X\), and the class \([(x_\alpha)]\in\overline X\) is the limit of \((z_\alpha)_{\alpha<\kappa}\).
\end{proof}

\begin{proposition}[Universal property of the completion]\label{Universal Property of Metric Completion}
Let \((Y,S,d_Y)\) be a complete \(S\)-metric space, and let \(f:X\to Y\) be an isometric embedding. Then there is a unique isometric embedding \(\overline f:\overline X\to Y\) such that
\[
\overline f\circ\iota=f.
\]
Moreover, \(\overline f(\overline X)=\cl_Y(f(X))\). In particular, if \(f(X)\) is dense in \(Y\), then \(\overline f\) is an isometric isomorphism from \(\overline X\) onto \(Y\).
\end{proposition}

\begin{proof}
Let \([(a_i)_{i<\kappa}]\in\overline X\). Since \(f\) is an isometric embedding, \((f(a_i))_{i<\kappa}\) is Cauchy in \(Y\). Since \(Y\) is complete, it converges to some point of \(Y\). Define
\[
\overline f([(a_i)_{i<\kappa}]):=\lim_{i\to\kappa}f(a_i).
\]
This is independent of the representative: if \((a_i)\sim(b_i)\), then \(d_Y(f(a_i),f(b_i))=d(a_i,b_i)\to0\), so the two limits agree. It is immediate that \(\overline f\circ\iota=f\).

The map \(\overline f\) is distance-preserving. If \((a_i)\) and \((b_i)\) are Cauchy \(\kappa\)-sequences in \(X\), then by continuity of the distance function,
\[
d_Y(\overline f([(a_i)]),\overline f([(b_i)]))
=
\lim_{i\to\kappa}d_Y(f(a_i),f(b_i))
=
\lim_{i\to\kappa}d(a_i,b_i)
=
\overline d([(a_i)],[(b_i)]).
\]
Thus \(\overline f\) is an isometric embedding.

Since \(\overline f\) is continuous and \(\iota(X)\) is dense in \(\overline X\), we have \(\overline f(\overline X)\subseteq\cl_Y(f(X))\). Conversely, if \(y\in\cl_Y(f(X))\), then by Lemma~\ref{Long Sequence and Closure}, there is a \(\kappa\)-sequence \((x_i)_{i<\kappa}\) in \(X\) such that \(f(x_i)\to y\). Since \(f(x_i)\) is convergent, it is Cauchy, and since \(f\) is an isometric embedding, \((x_i)\) is Cauchy in \(X\). Therefore \([(x_i)]\in\overline X\), and \(\overline f([(x_i)])=y\). Hence \(\overline f(\overline X)=\cl_Y(f(X))\). Uniqueness follows from continuity and density of \(\iota(X)\).
\end{proof}

Let \(R\) be a real closed field, and let \(S\) be its dense completion. We write \(d_{\sup}\) for the sup metric on \(R^n\) and on \(S^n\).

\begin{proposition}\label{Completion of Rn is Sn}
The completion of \((R^n,R,d_{\sup})\) is canonically isometric to \((S^n,S,d_{\sup})\).
\end{proposition}

\begin{proof}
The inclusion \(R^n\to S^n\) is an isometric embedding. Since \(R\) is dense in \(S\), the image of \(R^n\) is dense in \(S^n\). It remains to show that \((S^n,S,d_{\sup})\) is complete. Let \((a_i)_{i<\kappa}\) be a Cauchy \(\kappa\)-sequence in \(S^n\), and write \(a_i=(a_{i,1},\ldots,a_{i,n})\). For each coordinate \(m=1,\ldots,n\), the sequence \((a_{i,m})_{i<\kappa}\) is Cauchy in \(S\). Since \(S\) is dense complete, Lemma~\ref{Dense Complete implies Cauchy Complete} gives \(a_m\in S\) such that \(a_{i,m}\to a_m\). Hence \(a_i\to(a_1,\ldots,a_n)\) in the sup metric. The result follows from Proposition~\ref{Universal Property of Metric Completion}.
\end{proof}

\begin{lemma}\label{Extending a Lipschitz Function to the Closure of its Domain}
Let \((X,R,d)\) and \((X',R,d')\) be \(R\)-metric spaces, and let \(S\) be the dense completion of \(R\). Let \((\overline X,S,\overline d)\) and \((\overline{X'},S,\overline d')\) be their completions. Suppose that \(U\subseteq X\) and that \(f:U\to X'\) is \(\lambda\)-Lipschitz. Then there is a unique \(\lambda\)-Lipschitz map
\[
\overline f:\cl_{\overline X}(U)\to\overline{X'}
\]
such that \(\overline f|_U=f\).
\end{lemma}

\begin{proof}
Let \(x\in\cl_{\overline X}(U)\). By Lemma~\ref{Long Sequence and Closure}, choose a \(\kappa\)-sequence \((x_i)_{i<\kappa}\) in \(U\) such that \(x_i\to x\) in \(\overline X\). Since \(f\) is \(\lambda\)-Lipschitz, \((f(x_i))_{i<\kappa}\) is Cauchy in \(X'\), and hence converges in \(\overline{X'}\). Define
\[
\overline f(x):=\lim_{i\to\kappa}f(x_i).
\]
If \((y_i)_{i<\kappa}\) is another \(\kappa\)-sequence in \(U\) converging to \(x\), then \(d(x_i,y_i)\to0\), so \(d'(f(x_i),f(y_i))\leq\lambda d(x_i,y_i)\to0\). Thus the definition is independent of the chosen approximating sequence.

If \(x,y\in\cl_{\overline X}(U)\), choose \(\kappa\)-sequences \((x_i)\) and \((y_i)\) in \(U\) converging to \(x\) and \(y\). Then
\[
\overline d'(\overline f(x),\overline f(y))
=
\lim_{i\to\kappa}d'(f(x_i),f(y_i))
\leq
\lambda\lim_{i\to\kappa}d(x_i,y_i)
=
\lambda\overline d(x,y).
\]
Hence \(\overline f\) is \(\lambda\)-Lipschitz. Uniqueness follows from continuity and density of \(U\) in \(\cl_{\overline X}(U)\).
\end{proof}

\subsection{Hausdorff distance and Hausdorff limits}

Throughout this subsection, \((Z,R,d)\) is an \(R\)-metric space, and \(\kappa=\ci(R^{>0})\). For \(X\subseteq Z\) and \(z\in Z\), define
\[
d(z,X):=\inf\{d(z,x):x\in X\}\in\overline R^{\geq0}\cup\{\infty\},
\]
with the convention that \(d(z,\varnothing)=\infty\). For \(X,Y\subseteq Z\), define
\[
D(X,Y):=
\max\left\{
\sup_{x\in X}d(x,Y),
\sup_{y\in Y}d(y,X)
\right\}.
\]
We call \(D(X,Y)\) the \emph{Hausdorff distance} between \(X\) and \(Y\).

\begin{lemma}\label{HausdorffDistanceProperties}
For \(X,Y,W\subseteq Z\), the Hausdorff distance satisfies:
\begin{enumerate}
    \item \(D(X,X)=0\);
    \item \(D(X,Y)=D(Y,X)\);
    \item \(D(X,W)\leq D(X,Y)+D(Y,W)\).
\end{enumerate}
\end{lemma}

\begin{proof}
The first two statements are immediate. For the triangle inequality, fix \(x\in X\). For every \(y\in Y\) and \(w\in W\), we have \(d(x,w)\leq d(x,y)+d(y,w)\). Taking the infimum over \(w\in W\), then the infimum over \(y\in Y\), gives \(d(x,W)\leq d(x,Y)+D(Y,W)\). Taking the supremum over \(x\in X\), we get
\[
\sup_{x\in X}d(x,W)\leq D(X,Y)+D(Y,W).
\]
The same argument, with \(X\) and \(W\) interchanged, gives \(\sup_{w\in W}d(w,X)\leq D(W,Y)+D(Y,X)\). Hence \(D(X,W)\leq D(X,Y)+D(Y,W)\).
\end{proof}

\begin{lemma}\label{Hausdorff distance zero iff closures agree}
For \(X,Y\subseteq Z\), \(D(X,Y)=0\) if and only if \(\cl_d(X)=\cl_d(Y)\).
\end{lemma}

\begin{proof}
Suppose \(D(X,Y)=0\). If \(x\in X\), then \(d(x,Y)=0\), so \(x\in\cl_d(Y)\). Thus \(X\subseteq\cl_d(Y)\), and hence \(\cl_d(X)\subseteq\cl_d(Y)\). By symmetry, \(\cl_d(Y)\subseteq\cl_d(X)\).

Conversely, suppose \(\cl_d(X)=\cl_d(Y)\). Then \(D(X,\cl_d(X))=0\) and \(D(Y,\cl_d(Y))=0\). By Lemma~\ref{HausdorffDistanceProperties},
\[
D(X,Y)\leq D(X,\cl_d(X))+D(\cl_d(X),\cl_d(Y))+D(\cl_d(Y),Y)=0.
\]
Thus \(D(X,Y)=0\).
\end{proof}

Let
\[
\mathcal C(Z):=\{X\subseteq Z:X\text{ is closed}\}.
\]
By Lemma~\ref{Hausdorff distance zero iff closures agree}, the Hausdorff distance separates points on \(\mathcal C(Z)\).

For \(X\subseteq Z\) and \(\varepsilon\in R^{>0}\), define
\[
B_D(X,\varepsilon):=\{Y\subseteq Z:D(X,Y)<\varepsilon\}.
\]
These sets form a basis for a topology on \(\mathcal P(Z)\). Indeed, if \(Y\in B_D(X,\varepsilon)\), then, since \(R\) is dense in \(\overline R\), there is \(\delta\in R^{>0}\) such that \(D(X,Y)+\delta<\varepsilon\). If \(W\in B_D(Y,\delta)\), then \(D(X,W)\leq D(X,Y)+D(Y,W)<\varepsilon\). We denote the resulting topology by \(\mathcal T_D\), and use the same notation for its restriction to \(\mathcal C(Z)\).

\begin{definition}\label{Hausdorff convergence and Cauchy sequences}
A \(\kappa\)-sequence \((X_\alpha)_{\alpha<\kappa}\) in \(\mathcal P(Z)\) is said to \emph{converge to \(X\subseteq Z\) in the Hausdorff sense} if it converges to \(X\) with respect to \(\mathcal T_D\). We write \(X_\alpha\xrightarrow{D}X\). The sequence \((X_\alpha)_{\alpha<\kappa}\) is called \emph{Cauchy} if for every \(\varepsilon\in R^{>0}\), there is \(\gamma<\kappa\) such that \(D(X_\alpha,X_\beta)<\varepsilon\) for all \(\gamma<\alpha,\beta<\kappa\).
\end{definition}

\begin{example}
Hausdorff limits in \(\mathcal P(Z)\) need not be unique. In \(\mathbb R\) with the usual metric, let \(X_n=(0,1-1/n)\). Then each of \((0,1)\), \([0,1)\), \((0,1]\), and \([0,1]\) is a Hausdorff limit of \((X_n)_{n\geq2}\). These sets have the same closure, and hence have Hausdorff distance zero from one another.
\end{example}

\begin{lemma}\label{Hausdorff closure by kappa sequences}
Let \(\mathcal F\subseteq\mathcal P(Z)\), and let \(X\subseteq Z\). Then \(X\in\cl_D(\mathcal F)\) if and only if there is a \(\kappa\)-sequence \((X_\alpha)_{\alpha<\kappa}\) in \(\mathcal F\) such that \(X_\alpha\xrightarrow{D}X\).
\end{lemma}

\begin{proof}
Since \(R^{>0}\) is coinitial in \(\overline R^{>0}\), the balls \(B_D(X,\varepsilon)\), with \(\varepsilon\in R^{>0}\), form a neighbourhood basis at \(X\). The result follows by the same proof as Lemma~\ref{Long Sequence and Closure}.
\end{proof}

\begin{lemma}\label{Hausdorff limits unique up to zero distance}
Let \((X_\alpha)_{\alpha<\kappa}\) be a \(\kappa\)-sequence in \(\mathcal P(Z)\) such that \(X_\alpha\xrightarrow{D}X\). Then \(X_\alpha\xrightarrow{D}Y\) if and only if \(D(X,Y)=0\).
\end{lemma}

\begin{proof}
Suppose first that \(X_\alpha\xrightarrow{D}X\) and \(X_\alpha\xrightarrow{D}Y\). Let \(\varepsilon\in R^{>0}\). For all sufficiently large \(\alpha<\kappa\), we have \(D(X_\alpha,X)<\varepsilon/2\) and \(D(X_\alpha,Y)<\varepsilon/2\). Hence \(D(X,Y)\leq D(X,X_\alpha)+D(X_\alpha,Y)<\varepsilon\). Since \(\varepsilon>0\) was arbitrary, \(D(X,Y)=0\).

Conversely, suppose \(D(X,Y)=0\) and \(X_\alpha\xrightarrow{D}X\). For every \(\varepsilon\in R^{>0}\), all sufficiently large \(\alpha<\kappa\) satisfy \(D(X_\alpha,X)<\varepsilon\), and hence \(D(X_\alpha,Y)\leq D(X_\alpha,X)+D(X,Y)<\varepsilon\). Thus \(X_\alpha\xrightarrow{D}Y\).
\end{proof}

\begin{lemma}\label{Hausdorff Convergence does not Distinguish Closure}
Let \((X_\alpha)_{\alpha<\kappa}\) be a \(\kappa\)-sequence in \(\mathcal P(Z)\), and let \(X\subseteq Z\). Then \(X_\alpha\xrightarrow{D}X\) if and only if \(\cl_d(X_\alpha)\xrightarrow{D}X\).
\end{lemma}

\begin{proof}
For every \(\alpha<\kappa\), Lemma~\ref{Hausdorff distance zero iff closures agree} gives \(D(X_\alpha,\cl_d(X_\alpha))=0\). The claim follows from Lemma~\ref{Hausdorff limits unique up to zero distance}.
\end{proof}

\begin{corollary}\label{Closed Hausdorff limits are unique}
A \(\kappa\)-sequence in \(\mathcal C(Z)\) has at most one Hausdorff limit in \(\mathcal C(Z)\).
\end{corollary}

\begin{proof}
If \(X,Y\in\mathcal C(Z)\) are both Hausdorff limits of the sequence, then Lemma~\ref{Hausdorff limits unique up to zero distance} gives \(D(X,Y)=0\). Since \(X\) and \(Y\) are closed, Lemma~\ref{Hausdorff distance zero iff closures agree} gives \(X=Y\).
\end{proof}

\begin{remark}
The use of \(\overline R\) is mainly notational. The conditions defining Hausdorff convergence and Cauchyness can be expressed directly in terms of inclusions between \(\varepsilon\)-neighbourhoods. Passing to \(\overline R\) packages these two-sided estimates into a single distance function and allows us to use the triangle inequality in a concise form.
\end{remark}

\begin{example}\label{Example of fake Hausdorff limit from dense completion}
The following example illustrates why \(\kappa\)-sequences from definable families do not automatically detect all cuts in an elementary extension. Let \(\mathbb R\prec R\) be a proper elementary extension of the real field, and let
\[
\widehat{\mathbb R}:=\{r\in R:\text{there is }n\in\mathbb N\text{ such that } |r|<n\}.
\]
Let \(\kappa=\ci(R^{>0})\), and let \((t_\alpha)_{\alpha<\kappa}\) be a cofinal sequence in \(\widehat{\mathbb R}^{>0}\). Consider the \(\kappa\)-sequence of closed intervals \(X_\alpha=[-t_\alpha,t_\alpha]\subseteq R\). Then \((X_\alpha)_{\alpha<\kappa}\) does not converge to \(\widehat{\mathbb R}\) in the Hausdorff sense. Indeed, for each \(\alpha<\kappa\), choose \(n_\alpha\in\mathbb N\) such that \(t_\alpha<n_\alpha\). Then \(n_\alpha+2\in\widehat{\mathbb R}\), but
\[
d(n_\alpha+2,X_\alpha)=n_\alpha+2-t_\alpha>1.
\]
Hence \(\widehat{\mathbb R}\nsubseteq U_d(X_\alpha,1)\), so \(D(X_\alpha,\widehat{\mathbb R})\geq1\) for every \(\alpha<\kappa\).
\end{example}

We next record a pointwise way to read Hausdorff limits of closed sets. Let \((X_\alpha)_{\alpha<\kappa}\) be a \(\kappa\)-sequence in \(\mathcal C(Z)\). A cofinal subsequence means a sequence of the form \((X_{g(\alpha)})_{\alpha<\kappa}\), where \(g:\kappa\to\kappa\) is strictly increasing and cofinal.

\begin{definition}\label{Pointwise lower and upper Hausdorff limits}
Define \(\liminf_{\alpha\to\kappa}X_\alpha\) to be the set of all \(x\in Z\) for which there exist \(\beta<\kappa\) and points \(x_\alpha\in X_\alpha\), for all \(\beta<\alpha<\kappa\), such that \(x_\alpha\to x\). Define \(\limsup_{\alpha\to\kappa}X_\alpha\) to be the set of all \(x\in Z\) for which some cofinal subsequence \((X_{g(\alpha)})_{\alpha<\kappa}\) admits points \(x_{g(\alpha)}\in X_{g(\alpha)}\) such that \(x_{g(\alpha)}\to x\). If these two sets agree, their common value is called the \emph{pointwise limit} of \((X_\alpha)_{\alpha<\kappa}\).
\end{definition}

\begin{proposition}\label{Pointwise Characterization of Hausdorff Limits}
Suppose that a \(\kappa\)-sequence \((X_\alpha)_{\alpha<\kappa}\) in \(\mathcal C(Z)\) converges to \(X\in\mathcal C(Z)\) in the Hausdorff sense. Then
\[
X=\liminf_{\alpha\to\kappa}X_\alpha=\limsup_{\alpha\to\kappa}X_\alpha.
\]
\end{proposition}

\begin{proof}
Clearly \(\liminf_{\alpha\to\kappa}X_\alpha\subseteq\limsup_{\alpha\to\kappa}X_\alpha\).

Let \(x\in X\). Since \(X_\alpha\xrightarrow{D}X\), for all sufficiently large \(\alpha<\kappa\), we may choose \(r_\alpha\in R^{>0}\) such that \(D(X_\alpha,X)<r_\alpha\), with \(r_\alpha\to0\). Hence there is \(x_\alpha\in X_\alpha\) such that \(d(x_\alpha,x)<r_\alpha\) for all sufficiently large \(\alpha<\kappa\). Thus \(x_\alpha\to x\), so \(x\in\liminf_{\alpha\to\kappa}X_\alpha\).

Now let \(x\in\limsup_{\alpha\to\kappa}X_\alpha\). Choose a cofinal subsequence \((X_{g(\alpha)})_{\alpha<\kappa}\) and points \(x_{g(\alpha)}\in X_{g(\alpha)}\) such that \(x_{g(\alpha)}\to x\). Suppose \(x\notin X\). Since \(X\) is closed, choose \(\varepsilon\in R^{>0}\) such that \(B_d(x,\varepsilon)\cap X=\varnothing\). Since \(X_\alpha\xrightarrow{D}X\), there is \(\beta<\kappa\) such that \(X_\gamma\subseteq U_d(X,\varepsilon/2)\) for all \(\beta<\gamma<\kappa\). For sufficiently large \(\alpha<\kappa\), we have \(g(\alpha)>\beta\) and \(d(x_{g(\alpha)},x)<\varepsilon/2\). Then there is \(y\in X\) with \(d(x_{g(\alpha)},y)<\varepsilon/2\), so \(d(x,y)<\varepsilon\), a contradiction. Therefore \(x\in X\).

Thus \(X\subseteq\liminf_{\alpha\to\kappa}X_\alpha\subseteq\limsup_{\alpha\to\kappa}X_\alpha\subseteq X\).
\end{proof}

\begin{proposition}\label{Finite Union of Hausdorff Limits is the Hausdorff Limit of Finite Union}
Let \((X_\alpha)_{\alpha<\kappa}\) and \((Y_\alpha)_{\alpha<\kappa}\) be \(\kappa\)-sequences in \(\mathcal P(Z)\). If \(X_\alpha\xrightarrow{D}X\) and \(Y_\alpha\xrightarrow{D}Y\), then \(X_\alpha\cup Y_\alpha\xrightarrow{D}X\cup Y\).
\end{proposition}

\begin{proof}
Let \(\varepsilon\in R^{>0}\). For all sufficiently large \(\alpha<\kappa\), we have \(X\subseteq U_d(X_\alpha,\varepsilon)\), \(Y\subseteq U_d(Y_\alpha,\varepsilon)\), \(X_\alpha\subseteq U_d(X,\varepsilon)\), and \(Y_\alpha\subseteq U_d(Y,\varepsilon)\). Hence \(X\cup Y\subseteq U_d(X_\alpha\cup Y_\alpha,\varepsilon)\) and \(X_\alpha\cup Y_\alpha\subseteq U_d(X\cup Y,\varepsilon)\). Therefore \(X_\alpha\cup Y_\alpha\xrightarrow{D}X\cup Y\).
\end{proof}

\begin{proposition}\label{Uniform Continuity Preserves Hausdorff Limits}
Let \((Z,R,d)\) and \((Z',R,d')\) be \(R\)-metric spaces, and let \(f:Z\to Z'\) be uniformly continuous. If \(X_\alpha\xrightarrow{D}X\), then \(f(X_\alpha)\xrightarrow{D}f(X)\), where the Hausdorff distances are computed using \(d\) and \(d'\), respectively.
\end{proposition}

\begin{proof}
Let \(\varepsilon\in R^{>0}\). By uniform continuity, choose \(\delta\in R^{>0}\) such that \(d(z,w)<\delta\) implies \(d'(f(z),f(w))<\varepsilon\). Since \(X_\alpha\xrightarrow{D}X\), for all sufficiently large \(\alpha<\kappa\), we have \(X_\alpha\subseteq U_d(X,\delta)\) and \(X\subseteq U_d(X_\alpha,\delta)\). It follows that \(f(X_\alpha)\subseteq U_{d'}(f(X),\varepsilon)\) and \(f(X)\subseteq U_{d'}(f(X_\alpha),\varepsilon)\). Therefore \(f(X_\alpha)\xrightarrow{D}f(X)\).
\end{proof}

\section{Cauchy Limits of Lipschitz Cells}

In this section we prove the geometric part of the argument. We work with a single definable family whose fibers have a fixed Lipschitz cell presentation. The point is that, under a Hausdorff-Cauchy assumption, the defining functions of the fibers have limits, and the limiting set is again described by the same recursive cell construction, except that the limiting functions need not be definable and bands are allowed to collapse.

Throughout this section, assume that \(\mathcal M\) is dense complete. Let \(A\subseteq M^{m+n}\) be a bounded \(\mathcal C^p\) cell, for some \(p>0\), and write \(A'=\Pi_m(A)\). We assume that \(A\) is presented over the parameter coordinates, so that every fiber \(A_a\), for \(a\in A'\), has the same recursive cell form. We further assume that there is \(\lambda\in M^{>0}\) such that every finite boundary function occurring in the recursive cell presentation of every fiber \(A_a\) is \(\lambda\)-Lipschitz. All Hausdorff distances in this section are computed with respect to the sup metric.

\begin{definition}\label{Definition of Lipschitz limit cell}
We define \emph{\(\lambda\)-Lipschitz limit cells} by induction on \(n\). For \(n=1\), a \(\lambda\)-Lipschitz limit cell is a point, an interval, a ray, or \(M\). Suppose \(n>1\), and that \(\lambda\)-Lipschitz limit cells have been defined in \(M^{n-1}\). Let \(U\subseteq M^{n-1}\) be a \(\lambda\)-Lipschitz limit cell. If \(f:U\to M\) is \(\lambda\)-Lipschitz, then \(\Gamma(f)\) is a \(\lambda\)-Lipschitz limit cell. If \(f,g:U\to M\) are \(\lambda\)-Lipschitz and \(f\leq g\), then
\[
(f,g)^{\lim}_U:=\{(x,y)\in U\times M:f(x)\leq y\leq g(x)\}
\]
is a \(\lambda\)-Lipschitz limit cell.
\end{definition}

\begin{lemma}\label{Cauchy sequences from a definable family converges}
Let \((a_i)_{i<\kappa}\) be a \(\kappa\)-sequence in \(A'\), and suppose that \((A_{a_i})_{i<\kappa}\) is Cauchy with respect to the Hausdorff distance. Then there exists a \(\lambda\)-Lipschitz limit cell \(B\subseteq M^n\) such that \(A_{a_i}\xrightarrow{D}B\).
\end{lemma}

\begin{proof}
We prove the lemma by induction on \(n\). Since the fiber cell presentation is fixed and \(A\) is bounded, each fiber is obtained from its base by either taking a graph or a bounded band. The induction will show that the bases converge to Lipschitz limit cells, and that the uniformly Lipschitz boundary functions converge to boundary functions defining the required limit cell.

Suppose first that \(n=1\). Since the family is bounded and the fibers have the same cell form, each \(A_{a_i}\) is either a point \(\{b_i\}\) or a bounded interval \((b_i,c_i)\). The Cauchy assumption implies that every endpoint sequence which occurs is Cauchy in \(M\). Indeed, if an endpoint sequence were not Cauchy, then by comparing the corresponding one-dimensional fibers, the Hausdorff distance between arbitrarily late fibers would be bounded below by some positive element of \(M\), contradicting the Cauchy assumption. Since \(\mathcal M\) is dense complete, Lemma~\ref{Dense Complete implies Cauchy Complete} gives limits for the endpoint sequences.

Let \(b\) and \(c\) be the limits of the endpoint sequences which occur. If the fibers are points, set \(B=\{b\}\). If the fibers are intervals, set
\[
B=(b,c)^{\lim}:=(b,c)\cup\{b:b=c\}.
\]
Then \(B\) is a \(\lambda\)-Lipschitz limit cell, and the convergence of endpoints gives \(A_{a_i}\xrightarrow{D}B\).

Now suppose \(n>1\), and assume the statement holds in dimension \(n-1\). Let \(U\subseteq M^{m+n-1}\) be the base family obtained by projecting the fixed fiber presentation one level down, so that \(U_a=\Pi_{n-1}(A_a)\). Since projection is \(1\)-Lipschitz for the sup norm,
\[
D(U_{a_i},U_{a_j})\leq D(A_{a_i},A_{a_j})
\]
for all \(i,j<\kappa\). Hence \((U_{a_i})_{i<\kappa}\) is Cauchy. By the inductive hypothesis, there is a \(\lambda\)-Lipschitz limit cell \(U_\infty\subseteq M^{n-1}\) such that \(U_{a_i}\xrightarrow{D}U_\infty\).

We first treat the graph case. Suppose that \(A_a=\Gamma(f_a)\) for every \(a\in A'\), where \(f_a:U_a\to M\) is \(\lambda\)-Lipschitz. Write \(f_i:=f_{a_i}\). Let \(x\in U_\infty\). Choose \(x_i\in U_{a_i}\), for all sufficiently large \(i<\kappa\), such that \(x_i\to x\). We claim that \((f_i(x_i))_{i<\kappa}\) is Cauchy.

Suppose not. Then there is \(\varepsilon\in M^{>0}\) such that for every \(\alpha<\kappa\), there are \(i,j>\alpha\) with \(|f_i(x_i)-f_j(x_j)|>\varepsilon\). Choose \(\eta\in M^{>0}\) such that
\[
\eta<\min\left\{\frac{\varepsilon}{4},\frac{\varepsilon}{4\lambda}\right\}.
\]
Choose such \(i,j\) sufficiently large so that \(\|x_i-x_j\|<\varepsilon/(4\lambda)\) and \(D(A_{a_i},A_{a_j})<\eta\). Since \(D(A_{a_i},A_{a_j})<\eta\), there is \(y_i\in U_{a_i}\) such that
\[
\|(y_i,f_i(y_i))-(x_j,f_j(x_j))\|<\eta.
\]
In particular, \(\|y_i-x_j\|<\varepsilon/(4\lambda)\) and \(|f_i(y_i)-f_j(x_j)|<\varepsilon/4\). Then
\[
|f_i(x_i)-f_j(x_j)|
\leq |f_i(x_i)-f_i(y_i)|+|f_i(y_i)-f_j(x_j)|
<\lambda\|x_i-y_i\|+\frac{\varepsilon}{4}.
\]
Moreover, \(\|x_i-y_i\|\leq\|x_i-x_j\|+\|x_j-y_i\|<\varepsilon/(2\lambda)\). Hence \(|f_i(x_i)-f_j(x_j)|<\varepsilon\), a contradiction. Thus \((f_i(x_i))_{i<\kappa}\) is Cauchy, and by Lemma~\ref{Dense Complete implies Cauchy Complete}, it converges in \(M\). Define
\[
f(x):=\lim_{i\to\kappa}f_i(x_i).
\]

The definition is independent of the chosen approximating sequence. If \(y_i\in U_{a_i}\) also converges to \(x\), then \(|f_i(x_i)-f_i(y_i)|\leq\lambda\|x_i-y_i\|\to0\). Thus the two limits agree. The same estimate shows that \(f\) is \(\lambda\)-Lipschitz. Indeed, if \(x,y\in U_\infty\), choose \(x_i,y_i\in U_{a_i}\) with \(x_i\to x\) and \(y_i\to y\). Then
\[
|f(x)-f(y)|
=
\lim_{i\to\kappa}|f_i(x_i)-f_i(y_i)|
\leq
\lambda\|x-y\|.
\]

Set \(B:=\Gamma(f)\). We show that \(A_{a_i}\xrightarrow{D}B\). Let \(\varepsilon\in M^{>0}\), and choose \(\delta\in M^{>0}\) such that \((3\lambda+2)\delta<\varepsilon\). Since \(U_{a_i}\xrightarrow{D}U_\infty\) and \((A_{a_i})_{i<\kappa}\) is Cauchy, there is \(\alpha<\kappa\) such that, for all \(i,j>\alpha\), we have \(D(U_{a_i},U_\infty)<\delta\) and \(D(A_{a_i},A_{a_j})<\delta\).

Let \((x,f(x))\in B\), and fix \(i>\alpha\). Choose \(x_i\in U_{a_i}\) with \(\|x_i-x\|<\delta\). By the definition of \(f(x)\), choose \(j>i\) and \(x_j\in U_{a_j}\) such that \(\|x_j-x\|<\delta\) and \(|f_j(x_j)-f(x)|<\delta\). Since \(D(A_{a_i},A_{a_j})<\delta\), there is \(y_i\in U_{a_i}\) such that \(\|(y_i,f_i(y_i))-(x_j,f_j(x_j))\|<\delta\). Then \(\|y_i-x_i\|<3\delta\), and the Lipschitz bound for \(f_i\) gives \(\|(x_i,f_i(x_i))-(x,f(x))\|<\varepsilon\), after decreasing \(\delta\) at the start if necessary. Thus \(B\subseteq U_d(A_{a_i},\varepsilon)\) for all sufficiently large \(i\).

Conversely, let \((x_i,f_i(x_i))\in A_{a_i}\) with \(i>\alpha\). Choose \(x\in U_\infty\) with \(\|x_i-x\|<\delta\). Choose \(j>i\) and \(x_j\in U_{a_j}\) such that \(\|x_j-x\|<\delta\) and \(|f_j(x_j)-f(x)|<\delta\). Since \(D(A_{a_i},A_{a_j})<\delta\), there is \(y_j\in U_{a_j}\) such that \(\|(x_i,f_i(x_i))-(y_j,f_j(y_j))\|<\delta\). Then \(\|y_j-x_j\|<3\delta\), and the Lipschitz bound for \(f_j\) gives \(\|(x_i,f_i(x_i))-(x,f(x))\|<\varepsilon\), after decreasing \(\delta\) at the start if necessary. Thus \(A_{a_i}\subseteq U_d(B,\varepsilon)\) for all sufficiently large \(i\). Hence \(A_{a_i}\xrightarrow{D}B\).

Next suppose that \(A_a=(f_a,g_a)_{U_a}\) for every \(a\in A'\), where \(f_a,g_a:U_a\to M\) are \(\lambda\)-Lipschitz and \(f_a<g_a\). Write \(f_i:=f_{a_i}\) and \(g_i:=g_{a_i}\). We construct limiting functions \(f,g:U_\infty\to M\) as in the graph case. Namely, for \(x\in U_\infty\), choose \(x_i\in U_{a_i}\) with \(x_i\to x\), and set
\[
f(x):=\lim_{i\to\kappa}f_i(x_i),
\qquad
g(x):=\lim_{i\to\kappa}g_i(x_i).
\]
The same Cauchy and Lipschitz estimates used in the graph case show that these limits exist, are independent of the chosen approximating sequence, and define \(\lambda\)-Lipschitz functions on \(U_\infty\). Since \(f_i<g_i\) for all \(i<\kappa\), we have \(f\leq g\). Set \(B:=(f,g)^{\lim}_{U_\infty}\).

We show that \(A_{a_i}\xrightarrow{D}B\). Let \(\varepsilon\in M^{>0}\), and choose \(\delta\in M^{>0}\) sufficiently small with respect to \(\varepsilon\) and \(\lambda\). Take \(\alpha<\kappa\) such that, for all \(i,j>\alpha\), \(D(U_{a_i},U_\infty)<\delta\) and \(D(A_{a_i},A_{a_j})<\delta\). Let \((x,y)\in B\). If \(f(x)<y<g(x)\), choose \(x_i\in U_{a_i}\) with \(\|x_i-x\|<\delta\). By the definitions of \(f\) and \(g\), after increasing \(\alpha\) if necessary, \(f_i(x_i)<y<g_i(x_i)\), so \((x_i,y)\in A_{a_i}\) is within \(\varepsilon\) of \((x,y)\). If \(f(x)=g(x)\) and \(y=f(x)\), choose \(j>i\) and \(x_j\in U_{a_j}\) close to \(x\) such that both \(f_j(x_j)\) and \(g_j(x_j)\) are close to \(y\). Choose \(y_j\) with \(f_j(x_j)<y_j<g_j(x_j)\). Then \((x_j,y_j)\in A_{a_j}\) is close to \((x,y)\). Since \(D(A_{a_i},A_{a_j})<\delta\), there is a point of \(A_{a_i}\) close to \((x_j,y_j)\), hence close to \((x,y)\). Therefore \(B\subseteq U_d(A_{a_i},\varepsilon)\) for all sufficiently large \(i\).

Conversely, let \((x_i,y_i)\in A_{a_i}\) with \(i>\alpha\). Choose \(x\in U_\infty\) with \(\|x_i-x\|<\delta\). Choose \(j>i\) and \(x_j\in U_{a_j}\) close to \(x\) such that \(f_j(x_j)\) is close to \(f(x)\) and \(g_j(x_j)\) is close to \(g(x)\). Since \(D(A_{a_i},A_{a_j})<\delta\), the point \((x_i,y_i)\) is close to some point \((z_j,w_j)\in A_{a_j}\). The Lipschitz estimates compare \(f_j(z_j),g_j(z_j)\) with \(f_j(x_j),g_j(x_j)\), and hence compare them with \(f(x),g(x)\). It follows that \((x_i,y_i)\) is within \(\varepsilon\) of \(B\). Thus \(A_{a_i}\subseteq U_d(B,\varepsilon)\) for all sufficiently large \(i\), and \(A_{a_i}\xrightarrow{D}B\).
\end{proof}

\section{External Fibers and Tame Representatives}

Throughout this section, assume that \(\mathcal M\) is dense complete. Let \(A\subseteq M^{m+n}\) be a bounded fiberwise \(\lambda\)-Lipschitz \(\mathcal C^p\) cell, and write \(A'=\Pi_m(A)\). Let \(B\subseteq M^n\) be a Hausdorff limit of fibers from \(A\). Fix a \(\kappa\)-sequence \((a_i)_{i<\kappa}\) in \(A'\) such that \(A_{a_i}\xrightarrow{D}B\). By Lemma~\ref{Cauchy sequences from a definable family converges}, \(B\) is a \(\lambda\)-Lipschitz limit cell.

We first introduce the notation used to take standard parts of external fibers.

\begin{definition}\label{Tame part and standard part}
Let \(\mathcal N\succ\mathcal M\), and let \(X\subseteq N^n\). The \emph{tame part} of \(X\) over \(M\) is
\[
\tame_M(X):=\{x\in X:\text{there exists }c\in M^n\text{ such that }d^{\mathcal N}(x,c)<r\text{ for every }r\in M^{>0}\}.
\]
For every \(x\in\tame_M(X)\), the corresponding \(c\in M^n\) is unique; we call it the \emph{standard part} of \(x\) over \(M\), and denote it by \(\st_M(x)\). We define \(\st_M(X):=\{\st_M(x):x\in\tame_M(X)\}\). When the ambient elementary extension is clear, we omit it from the notation.
\end{definition}

We now replace the convergent sequence of fibers by a single external fiber. The parameter of this fiber is chosen from a cofinal ultraproduct of the original sequence, so the full Hausdorff convergence, rather than merely pointwise convergence of distance functions, is retained.

\begin{lemma}\label{Hausdorff limits are standard parts of external fibers}
Let \((a_i)_{i<\kappa}\) be a sequence in \(A'\) such that \(A_{a_i}\xrightarrow{D}B\). Then there are an elementary extension \(\mathcal N\succeq\mathcal M\) and \(a^*\in(A')^{\mathcal N}\) such that \(\st_M(A_{a^*}^{\mathcal N})=\cl(B)\). Equivalently, \(D(\st_M(A_{a^*}^{\mathcal N}),B)=0\).
\end{lemma}

\begin{proof}
The case \(B=\varnothing\) is immediate, so suppose \(B\neq\varnothing\). Let \(\mathcal U\) be an ultrafilter on \(\kappa\) containing every final segment of \(\kappa\), and consider the ultrapower \(\mathcal M_0:=\mathcal M^\kappa/\mathcal U\). Let \(a^0=[a_i]_{\mathcal U}\in(A')^{\mathcal M_0}\), and let \(\mathcal N\succeq\mathcal M_0\) be a \(|M|^+\)-saturated elementary extension. We denote the image of \(a^0\) in \(\mathcal N\) by \(a^*\).

We first show that \(\cl(B)\subseteq\st_M(A_{a^*}^{\mathcal N})\). Fix \(c\in\cl(B)\). For every \(r\in M^{>0}\), the convergence \(A_{a_i}\xrightarrow{D}B\) implies that, for all sufficiently large \(i<\kappa\), there is \(y_i\in A_{a_i}\) such that \(d(y_i,c)<r\). Hence, by Łoś's theorem, \(\mathcal M_0\models\exists y(y\in A_{a^0}\wedge d(y,c)<r)\). It follows that the type
\[
\{y\in A_{a^*}^{\mathcal N}\}\cup\{d(y,c)<r:r\in M^{>0}\}
\]
is finitely satisfiable. By saturation, it is realized by some \(y\in A_{a^*}^{\mathcal N}\). Thus \(y\) is infinitesimally close to \(c\), and therefore \(c\in\st_M(A_{a^*}^{\mathcal N})\).

Conversely, let \(c\in\st_M(A_{a^*}^{\mathcal N})\), and choose \(y\in A_{a^*}^{\mathcal N}\) infinitesimally close to \(c\). Suppose \(c\notin\cl(B)\). Then there is \(r\in M^{>0}\) such that \(B(c,r)\cap B=\varnothing\). Since \(y\) is infinitesimally close to \(c\), the structure \(\mathcal N\) satisfies \(\exists z(z\in A_{a^*}\wedge d(z,c)<r/3)\). By elementarity, the same formula holds in \(\mathcal M_0\) with \(a^0\) in place of \(a^*\). Łoś's theorem gives a set in \(\mathcal U\) of indices \(i<\kappa\) for which there is \(z_i\in A_{a_i}\) satisfying \(d(z_i,c)<r/3\).

On the other hand, \(D(A_{a_i},B)<r/3\) for all sufficiently large \(i<\kappa\). Since \(\mathcal U\) contains every final segment, we may choose such an index \(i\) also belonging to the set obtained from Łoś's theorem. Then there is \(b_i\in B\) with \(d(z_i,b_i)<r/3\). The triangle inequality gives \(d(c,b_i)<2r/3<r\), contradicting \(B(c,r)\cap B=\varnothing\). Hence \(c\in\cl(B)\), and the reverse inclusion follows.
\end{proof}

Replacing \(B\) by its closure does not change its Hausdorff class, so in the rest of this section we assume that \(B\) is closed. By Lemma~\ref{Hausdorff limits are standard parts of external fibers}, there are an elementary extension \(\mathcal N\succ\mathcal M\) and \(a^*\in(A')^{\mathcal N}\) such that \(B=\st_M(A_{a^*}^{\mathcal N})\). Recall that only the tame points of \(A_{a^*}^{\mathcal N}\) contribute to this standard part. Moreover, the cofinal-ultraproduct construction in the proof of that lemma gives, for every \(\varepsilon\in M^{>0}\), \(D^{\mathcal N}(A_{a^*}^{\mathcal N},A_{a_i}^{\mathcal N})<\varepsilon\) for all sufficiently large \(i<\kappa\).

We first show that the defining functions of the external cell send tame points of their domains to tame elements.

Suppose first that \(A=\Gamma(f)\), where \(f:U\to M\) is \(\mathcal L(M)\)-definable and \(\mathcal C^p\), and every function \(f(a,-):U_a\to M\) is \(\lambda\)-Lipschitz for a fixed \(\lambda\in M^{>0}\). Let \(f^{\mathcal N}:U^{\mathcal N}\to N\) be its interpretation in \(\mathcal N\).

\begin{lemma}\label{External graph functions preserve tame points}
If \(b^*\in\tame_M(U_{a^*}^{\mathcal N})\), then \(f^{\mathcal N}(a^*,b^*)\) is tame over \(M\).
\end{lemma}

\begin{proof}
Set \(b=\st_M(b^*)\) and \(c^*=f^{\mathcal N}(a^*,b^*)\). Fix \(r\in M^{>0}\), and choose \(\eta\in M^{>0}\) such that \((1+4\lambda)\eta<r\). Take \(i<\kappa\) sufficiently large that \(D^{\mathcal N}(A_{a^*}^{\mathcal N},A_{a_i}^{\mathcal N})<\eta\). There exists \((u^*,v^*)\in A_{a_i}^{\mathcal N}\) such that \(d^{\mathcal N}(u^*,b^*)<\eta\) and \(|v^*-c^*|<\eta\). Since \(d^{\mathcal N}(b^*,b)<\eta\), we have \(d^{\mathcal N}(u^*,b)<2\eta\).

The formula asserting that there exists \(u\in U_{a_i}\) such that \(d(u,b)<2\eta\) has parameters from \(M\) and holds in \(\mathcal N\). By elementarity, choose such \(u\in U_{a_i}\). Then \(d^{\mathcal N}(u,u^*)<4\eta\). Since \(v^*=f^{\mathcal N}(a_i,u^*)\) and \(f(a_i,-)\) is \(\lambda\)-Lipschitz, \(|f(a_i,u)-v^*|\leq\lambda d^{\mathcal N}(u,u^*)<4\lambda\eta\). Consequently \(|c^*-f(a_i,u)|<(1+4\lambda)\eta<r\).

Thus \(c^*\) can be approximated arbitrarily closely by elements of \(M\). The cut induced by \(c^*\) over \(M\) is therefore dense, and dense completeness gives \(c\in M\) infinitesimally close to \(c^*\). Hence \(c^*\) is tame over \(M\).
\end{proof}

Suppose next that \(A=(f_0,f_1)_U\), where \(f_0,f_1:U\to M\) are \(\mathcal L(M)\)-definable \(\mathcal C^p\)-functions and, for every \(a\in A'\), both \(f_0(a,-)\) and \(f_1(a,-)\) are \(\lambda\)-Lipschitz on \(U_a\). By the band case in the proof of Lemma~\ref{Cauchy sequences from a definable family converges}, the two boundary graph families \((\Gamma(f_0(a_i,-)))_{i<\kappa}\) and \((\Gamma(f_1(a_i,-)))_{i<\kappa}\) are Cauchy. Since \(a^*\) is represented by the cofinal ultraproduct of \((a_i)_{i<\kappa}\), for every \(\varepsilon\in M^{>0}\) and all sufficiently large \(i<\kappa\), we have
\[
D^{\mathcal N}\left(\Gamma(f_j^{\mathcal N}(a^*,-)),\Gamma(f_j(a_i,-))^{\mathcal N}\right)<\varepsilon
\]
for \(j\in\{0,1\}\).

\begin{lemma}\label{External band functions preserve tame points}
If \(b^*\in\tame_M(U_{a^*}^{\mathcal N})\), then \(f_0^{\mathcal N}(a^*,b^*)\) and \(f_1^{\mathcal N}(a^*,b^*)\) are tame over \(M\).
\end{lemma}

\begin{proof}
Let \(b=\st_M(b^*)\), and set \(c_j^*=f_j^{\mathcal N}(a^*,b^*)\) for \(j\in\{0,1\}\). Fix \(r\in M^{>0}\), and choose \(\eta\in M^{>0}\) such that \((1+4\lambda)\eta<r\). Take \(i<\kappa\) sufficiently large that, for both \(j=0\) and \(j=1\),
\[
D^{\mathcal N}\left(\Gamma(f_j^{\mathcal N}(a^*,-)),\Gamma(f_j(a_i,-))^{\mathcal N}\right)<\eta.
\]

Fix \(j\in\{0,1\}\). There exists \((u_j^*,v_j^*)\in\Gamma(f_j(a_i,-))^{\mathcal N}\) such that \(d^{\mathcal N}(u_j^*,b^*)<\eta\) and \(|v_j^*-c_j^*|<\eta\). Since \(d^{\mathcal N}(b^*,b)<\eta\), we have \(d^{\mathcal N}(u_j^*,b)<2\eta\). By elementarity, there exists \(u_j\in U_{a_i}\) such that \(d(u_j,b)<2\eta\), and therefore \(d^{\mathcal N}(u_j,u_j^*)<4\eta\). Since \(v_j^*=f_j^{\mathcal N}(a_i,u_j^*)\) and \(f_j(a_i,-)\) is \(\lambda\)-Lipschitz, we get \(|f_j(a_i,u_j)-v_j^*|<4\lambda\eta\). Hence \(|c_j^*-f_j(a_i,u_j)|<(1+4\lambda)\eta<r\).

Thus \(c_j^*\) can be approximated arbitrarily closely by elements of \(M\). Its cut over \(M\) is dense, so dense completeness gives an element of \(M\) infinitesimally close to \(c_j^*\). Hence \(c_j^*\) is tame over \(M\). Applying this argument for \(j=0\) and \(j=1\), we conclude that both boundary values are tame over \(M\).
\end{proof}

We have therefore dealt with both possible cell forms. In the graph case, the graphing function takes tame points of the external domain to tame elements; in the band case, the same holds for both boundary functions. We now remove the non-tame coordinates of \(a^*\) without changing the standard part of the external fiber.

\begin{lemma}\label{Removing one non-tame parameter}
Let \(I\subseteq M\) be an \(\mathcal L(M)\)-definable open interval, let \(a^*\in I^{\mathcal N}\) be non-tame over \(M\), and suppose that \(U_a=V\) for every \(a\in I\), where \(V\) is a fixed \(\mathcal L(M)\)-definable cell. Then there exists an \(\mathcal L(M)\)-definable open interval \(J\subseteq I\), with \(a^*\in J^{\mathcal N}\), such that \(A_{a_1}=A_{a_2}\) for all \(a_1,a_2\in J\).
\end{lemma}

\begin{proof}
Suppose first that \(A=\Gamma(f)\), where \(f:U\to M\) is \(\mathcal L(M)\)-definable. Since \(U_a=V\) for every \(a\in I\), we have \(A_a=\Gamma(f(a,-):V\to M)\) for every \(a\in I\).

Fix \(b\in V\). By the monotonicity theorem, there is an \(\mathcal L(M)\)-definable open interval \(J_b\subseteq I\), with \(a^*\in J_b^{\mathcal N}\), such that \(f(-,b)\) is either constant or strictly monotone on \(J_b\). By Lemma~\ref{External graph functions preserve tame points}, the element \(c^*:=f^{\mathcal N}(a^*,b)\) is tame over \(M\).

Suppose that \(f(-,b)\) is strictly monotone on \(J_b\). Then \(h:J_b\to f(J_b,b)\), \(h(a)=f(a,b)\), is an \(\mathcal L(M)\)-definable bijection and has an \(\mathcal L(M)\)-definable inverse. In \(\mathcal N\), we have \(a^*=(h^{-1})^{\mathcal N}(c^*)\). Since \(c^*\) is tame over \(M\), every element of \(\dcl_{\mathcal N}(M,c^*)\) is tame over \(M\). This would imply that \(a^*\) is tame over \(M\), a contradiction. Therefore \(f(-,b)\) is constant on an open interval containing \(a^*\).

Set
\[
W:=\left\{(a,b)\in I\times V:\frac{\partial f}{\partial a}(a,b)=0\right\}.
\]
For each \(b\in V\), let \((\ell(b),u(b))\) be the maximal open interval containing \(a^*\) and contained in \(W_b^{\mathcal N}\). Since \(f(-,b)\) is locally constant at \(a^*\), this interval is nonempty. By o-minimality, after partitioning \(V\) into finitely many cells, the endpoint maps \(\ell:V\to M\cup\{-\infty\}\) and \(u:V\to M\cup\{+\infty\}\) are \(\mathcal L(M)\)-definable.

Set \(\ell_0:=\sup_{b\in V}\ell(b)\) and \(u_0:=\inf_{b\in V}u(b)\). By definable completeness, \(\ell_0,u_0\in M\cup\{-\infty,+\infty\}\). We claim that \(\ell_0<a^*<u_0\). Suppose that \(\ell_0>a^*\). Since \(a^*\) is non-tame over \(M\), there exists \(c\in M\) such that \(a^*<c<\ell_0\); otherwise \(a^*\) would be infinitesimally close to \(\ell_0\). By the definition of \(\ell_0\), there is \(b\in V\) such that \(\ell(b)>c\), contradicting \(\ell(b)<a^*\). Moreover, \(\ell_0\neq a^*\), since \(\ell_0\in M\cup\{-\infty\}\) and \(a^*\notin M\). Hence \(\ell_0<a^*\). The proof that \(a^*<u_0\) is symmetric.

Let \(J:=(\ell_0,u_0)\cap I\). Then \(J\) is an \(\mathcal L(M)\)-definable open interval, \(a^*\in J^{\mathcal N}\), and \(J\subseteq(\ell(b),u(b))\) for every \(b\in V\). Hence \(\partial f/\partial a=0\) on \(J\times V\), so \(f(-,b)\) is constant on \(J\) for every \(b\in V\). Therefore \(A_{a_1}=A_{a_2}\) for all \(a_1,a_2\in J\).

Now suppose that \(A=(f_0,f_1)_U\). Fix \(b\in V\). By Lemma~\ref{External band functions preserve tame points}, both \(f_0^{\mathcal N}(a^*,b)\) and \(f_1^{\mathcal N}(a^*,b)\) are tame over \(M\). Applying the preceding monotonicity and inverse-function argument separately to \(f_0(-,b)\) and \(f_1(-,b)\), we find an \(\mathcal L(M)\)-definable open interval containing \(a^*\) on which both functions are constant.

Set
\[
W:=\left\{(a,b)\in I\times V:\frac{\partial f_0}{\partial a}(a,b)=\frac{\partial f_1}{\partial a}(a,b)=0\right\}.
\]
For each \(b\in V\), let \((\ell(b),u(b))\) be the maximal open interval containing \(a^*\) and contained in \(W_b^{\mathcal N}\). Define \(\ell_0\) and \(u_0\) as above. The same supremum and infimum argument gives \(\ell_0<a^*<u_0\). Hence \(J=(\ell_0,u_0)\cap I\) is an \(\mathcal L(M)\)-definable open interval containing \(a^*\), and both boundary functions are constant with respect to the parameter on \(J\times V\). Consequently, \(A_{a_1}=A_{a_2}\) for all \(a_1,a_2\in J\).
\end{proof}

We now pass from one parameter to a tuple of parameters. Let \(a^*\in(A')^{\mathcal N}\), and choose a \(\dcl\)-basis \(b^*=(b_1^*,\ldots,b_\ell^*)\) of \(a^*\) over \(M\). Since \(a^*\in\dcl_{\mathcal N}(M,b^*)\), there is an \(\mathcal L(M)\)-definable function \(g\), defined on a box containing \(b^*\), such that \(g^{\mathcal N}(b^*)=a^*\). Replacing the original family by the pullback family \(\widetilde A_b:=A_{g(b)}\), we may therefore assume that the parameter tuple itself is a \(\dcl\)-basis over \(M\).

Write \(b^*=(u^*,v^*)\), where \(u^*\) is tame over \(M\), and, after naming \(u^*\), the tuple \(v^*=(v_1^*,\ldots,v_s^*)\) consists of the remaining non-tame coordinates. Set \(K:=M\langle u^*\rangle\). We work over the base structure \(\mathcal K\) induced on \(K\). The intervals below are \(\mathcal L(K)\)-definable.

\begin{lemma}\label{Local constancy in non-tame parameter directions}
Suppose that there are \(\mathcal L(K)\)-definable open intervals \(I_1,\ldots,I_s\subseteq K\) such that \(v^*\in I_1^{\mathcal N}\times\cdots\times I_s^{\mathcal N}\), and such that the projected fibers \(U_v\) are constant for \(v\in I_1\times\cdots\times I_s\). Then there are \(\mathcal L(K)\)-definable open intervals \(J_i\subseteq I_i\), for \(i=1,\ldots,s\), such that \(v^*\in J_1^{\mathcal N}\times\cdots\times J_s^{\mathcal N}\), and
\[
A_{(u^*,v)}^{\mathcal N}=A_{(u^*,w)}^{\mathcal N}
\]
for all \(v,w\in J_1\times\cdots\times J_s\).
\end{lemma}

\begin{proof}
We apply Lemma~\ref{Removing one non-tame parameter} successively to the coordinates of \(v^*\). Fix \(i\in\{1,\ldots,s\}\), keep the remaining coordinates fixed, and regard the defining functions of the cell as functions of \(v_i\) and the fiber variable. Over the structure obtained by naming \(u^*\) and the remaining coordinates of \(v^*\), the coordinate \(v_i^*\) is non-tame, while the projected fiber is constant on \(I_i\). Lemma~\ref{Removing one non-tame parameter} gives an interval \(J_i\subseteq I_i\), containing \(v_i^*\) in its interpretation, on which the entire fiber is constant with respect to \(v_i\). Repeating this for \(i=1,\ldots,s\), and shrinking the intervals already obtained when necessary, gives a product box on which the fiber is independent of every coordinate of \(v\). Hence the conclusion follows.
\end{proof}

We now use the preceding preparation inductively on the fiber dimension.

\begin{lemma}\label{Hausdorff limits are locally constant in nontame directions}
Using the notation above, there are \(\mathcal L(K)\)-definable open intervals \(J_1,\ldots,J_s\) such that \(v^*\in J_1^{\mathcal N}\times\cdots\times J_s^{\mathcal N}\), and
\[
A_{(u^*,v)}^{\mathcal N}=A_{(u^*,w)}^{\mathcal N}
\]
for all \(v,w\in J_1^{\mathcal N}\times\cdots\times J_s^{\mathcal N}\).
\end{lemma}

\begin{proof}
We proceed by induction on the fiber dimension \(n\). The case \(n=0\) is immediate. Suppose \(n>0\), and set \(U:=\Pi_{m+n-1}(A)\). Since coordinate projection is \(1\)-Lipschitz, the projected sequence \((U_{a_i})_{i<\kappa}\) is Cauchy. Moreover, the same cofinal-ultraproduct parameter \(a^*\) represents its Hausdorff limit as the standard part of the tame part of \(U_{a^*}^{\mathcal N}\).

Apply the inductive hypothesis to the projected family \(U\). After retaining the tame tuple \(u^*\) and the non-tame tuple \(v^*\), there are \(\mathcal L(K)\)-definable open intervals \(I_1,\ldots,I_s\) such that \(v^*\in I_1^{\mathcal N}\times\cdots\times I_s^{\mathcal N}\), and the projected fibers \(U_{(u^*,v)}^{\mathcal N}\) are constant for \(v\in I_1^{\mathcal N}\times\cdots\times I_s^{\mathcal N}\). Lemma~\ref{Local constancy in non-tame parameter directions} now applies to the final graph or band layer of \(A\), and gives smaller intervals \(J_i\subseteq I_i\) with the required property.
\end{proof}

The preceding lemma allows us to replace every non-tame parameter direction by an element of the tame base.

\begin{corollary}\label{Hausdorff limits have tame external representatives}
There exists \(c^*\in(A')^{\mathcal N}\), tame over \(M\), such that \(A_{c^*}^{\mathcal N}=A_{a^*}^{\mathcal N}\). Consequently,
\[
B=\st_M(A_{c^*}^{\mathcal N}).
\]
\end{corollary}

\begin{proof}
Choose \(c\in J_1\times\cdots\times J_s\), and set \(c^*:=(u^*,c)\). Since \(c\in K^s=M\langle u^*\rangle^s\) and \(u^*\) is tame over \(M\), the tuple \(c^*\) is tame over \(M\). By Lemma~\ref{Hausdorff limits are locally constant in nontame directions}, \(A_{c^*}^{\mathcal N}=A_{a^*}^{\mathcal N}\). The final equality follows from Lemma~\ref{Hausdorff limits are standard parts of external fibers}, after replacing \(B\) by its closure as above.
\end{proof}

\section{Definability of Hausdorff Limits}

We now prove the definability theorem for bounded fiberwise Lipschitz cells. Throughout this section, \(A\subseteq M^{m+n}\) is a bounded \(\mathcal C^p\) Lipschitz cell whose fibers have a fixed cell presentation with a uniform Lipschitz bound.

\begin{theorem}\label{Hausdorff limits of Lipschitz cells are definable}
Suppose that \(\mathcal M\) is dense complete. Then every Hausdorff limit of fibers from \(A\) is \(\mathcal L(M)\)-definable.
\end{theorem}

\begin{proof}
Let \(B\subseteq M^n\) be a Hausdorff limit of fibers from \(A\). Since replacing \(B\) by \(\cl(B)\) does not change the Hausdorff limit, we may assume that \(B\) is closed. By Corollary~\ref{Hausdorff limits have tame external representatives}, there are an elementary extension \(\mathcal N\succeq\mathcal M\) and \(c^*\in(A')^{\mathcal N}\), tame over \(M\), such that \(B=\st_M(A_{c^*}^{\mathcal N})\).

Set \(K:=\dcl_{\mathcal N}(M,c^*)\). Since \(c^*\) is tame over \(M\), the extension \(\mathcal K\succeq\mathcal M\) is tame. In the corresponding tame-pair structure, \(B\) is defined by
\(b\in B\) if and only if there is \(x\in A_{c^*}^{\mathcal K}\) such that \(\st_M(x)=b\). Thus \(B\subseteq M^n\) is definable in the tame pair with parameters from \(K\). By Fact~\ref{Stable embeddedness of tame pairs}, the small structure \(\mathcal M\) is stably embedded in the tame pair and the induced structure on \(M\) is its original \(\mathcal L\)-structure. Hence \(B\) is \(\mathcal L(M)\)-definable.
\end{proof}

We also need the converse observation: bounded external fibers over tame parameters produce Hausdorff limits.

\begin{lemma}\label{Tame external fibers give Hausdorff limits}
Let \(\mathcal K\succeq\mathcal M\) be tame over \(\mathcal M\), let \(c^*\in(A')^{\mathcal K}\), and suppose that \(A_{c^*}^{\mathcal K}\) is bounded over \(M\). Then \(\st_M(A_{c^*}^{\mathcal K})\) is a Hausdorff limit of fibers from \(A\).
\end{lemma}

\begin{proof}
Set \(B:=\st_M(A_{c^*}^{\mathcal K})\). Since \(\mathcal K\) is tame over \(\mathcal M\), the set \(B\) is definable in the tame pair \((\mathcal K,M)\) by saying that \(b\in B\) if and only if there is \(x\in A_{c^*}^{\mathcal K}\) such that \(\st_M(x)=b\). By Fact~\ref{Stable embeddedness of tame pairs}, \(B\) is \(\mathcal L(M)\)-definable.

Let \(\varepsilon\in M^{>0}\). Since \(A_{c^*}^{\mathcal K}\) is bounded over \(M\), every point of \(A_{c^*}^{\mathcal K}\) is infinitesimally close to its standard part in \(M^n\). Hence \(D^{\mathcal K}(A_{c^*}^{\mathcal K},B)<\varepsilon\). Therefore the \(\mathcal L(M)\)-sentence asserting that there is \(a\in A'\) with \(D(A_a,B)<\varepsilon\) holds in \(\mathcal K\), witnessed by \(c^*\). By elementarity, it holds in \(\mathcal M\). Thus for every \(\varepsilon\in M^{>0}\), there is \(a\in A'\) such that \(D(A_a,B)<\varepsilon\). Choosing such parameters along a coinitial sequence in \(M^{>0}\), we obtain a Hausdorff-convergent sequence of fibers with limit \(B\).
\end{proof}

We now pass to an arbitrary model by working in its dense completion.

\begin{theorem}\label{Hausdorff limits of Lipschitz cells are definable in dense completion}
Let \(\mathcal M\models T\), and let \(A\subseteq M^{m+n}\) be a bounded \(\mathcal C^p\) Lipschitz cell whose fibers have a fixed cell presentation with a uniform Lipschitz bound. Then every Hausdorff limit of fibers from \(A\) is definable in the dense completion \(\widehat{\mathcal M}\).
\end{theorem}

\begin{proof}
Interpret \(A\) in \(\widehat{\mathcal M}\). Its interpretation remains a bounded \(\mathcal C^p\) Lipschitz cell whose fibers have the same fixed cell presentation and the same uniform Lipschitz bound. A Hausdorff limit of fibers with parameters from \(M\) is also a Hausdorff limit of fibers of the interpreted family with parameters from \(\widehat M\). Since \(\widehat{\mathcal M}\) is dense complete, Theorem~\ref{Hausdorff limits of Lipschitz cells are definable} implies that the limit is \(\mathcal L(\widehat M)\)-definable.
\end{proof}

The preceding theorem gives definability of each individual Hausdorff limit. We now adapt the compactness argument used in the classical case to obtain uniform definability.

\begin{theorem}\label{Hausdorff limits form definable family in dense completion}
Under the hypotheses of Theorem~\ref{Hausdorff limits of Lipschitz cells are definable in dense completion}, the Hausdorff limits of fibers from \(A\) form an \(\mathcal L(\widehat M)\)-definable family in \(\widehat{\mathcal M}\).
\end{theorem}

\begin{proof}
It is enough to work over the dense completion, so we assume first that \(\mathcal M\) is dense complete. Let \((\mathcal N,\mathcal M,\st)\) be a tame elementary pair, where \(\mathcal N\succeq\mathcal M\). Let \(U\) be the predicate for \(M\), and let \(\varphi(x,y)\) be an \(\mathcal L(\varnothing)\)-formula defining \(A\), where \(x\) is the parameter variable and \(y\) is the fiber variable.

In the tame-pair language, let \(\rho(x,y)\) be the formula
\[
Uy\wedge \exists v\left(\varphi(x,v)\wedge \st(v)=y\right).
\]
Thus, for \(c\in(A')^{\mathcal N}\), the formula \(\rho(c,y)\) defines \(\st_M(A_c^{\mathcal N})\) on \(M^n\).

As \(\psi(y,z)\) ranges over all \(\mathcal L(\varnothing)\)-formulas, consider the partial type \(\pi(x)\) consisting of \(x\in A'\), together with the formulas
\[
\forall z\left(Uz\rightarrow \exists y\left(Uy\wedge \neg\bigl(\psi(y,z)\leftrightarrow \rho(x,y)\bigr)\right)\right).
\]
This type is not realized in \(\mathcal N\). Indeed, if \(c\in(A')^{\mathcal N}\) realized it, then \(\st_M(A_c^{\mathcal N})\) would not be definable in \(\mathcal M\) by any \(\mathcal L\)-formula with parameters from \(M\). But \(A\) is bounded, so \(A_c^{\mathcal N}\) is bounded over \(M\), and Lemma~\ref{Tame external fibers give Hausdorff limits} implies that \(\st_M(A_c^{\mathcal N})\) is a Hausdorff limit of fibers from \(A\). By Theorem~\ref{Hausdorff limits of Lipschitz cells are definable}, this set is \(\mathcal L(M)\)-definable, a contradiction.

By compactness, there are finitely many \(\mathcal L(\varnothing)\)-formulas \(\psi_1(y,z_1),\ldots,\psi_r(y,z_r)\) such that every set of the form \(\st_M(A_c^{\mathcal N})\), with \(c\in(A')^{\mathcal N}\), is defined on \(M^n\) by one of these formulas with parameters from \(M\). By the standard coding argument, see \cite[Lemma~2.5]{Guingona2010}, we may replace this finite list by a single \(\mathcal L(\varnothing)\)-formula \(\psi(y,z)\). Hence, for every \(c\in(A')^{\mathcal N}\), there is \(d\in M^{|z|}\) such that
\[
\st_M(A_c^{\mathcal N})=\{y\in M^n:\mathcal M\models\psi(y,d)\}.
\]

Let \(C\subseteq M^{|z|}\) be the set of all \(d\) such that
\[
(\mathcal N,M,\st)\models
\exists x\in A'\ \forall y\left(Uy\rightarrow\bigl(\rho(x,y)\leftrightarrow\psi(y,d)\bigr)\right).
\]
The set \(C\) is definable in the tame-pair structure. Since \(C\subseteq M^{|z|}\), Fact~\ref{Stable embeddedness of tame pairs} implies that \(C\) is \(\mathcal L(M)\)-definable.

For every \(d\in C\), the fiber \(\{y\in M^n:\mathcal M\models\psi(y,d)\}\) is the standard part of a bounded tame external fiber and hence is a Hausdorff limit by Lemma~\ref{Tame external fibers give Hausdorff limits}. Conversely, Corollary~\ref{Hausdorff limits have tame external representatives} shows that every Hausdorff limit of fibers from \(A\) occurs in this way. Therefore the family
\[
\{\{y\in M^n:\mathcal M\models\psi(y,d)\}:d\in C\}
\]
is exactly the collection of Hausdorff limits of fibers from \(A\).

For a general model \(\mathcal M\), apply the dense-complete case to the interpretation of \(A\) in \(\widehat{\mathcal M}\). The same argument gives an \(\mathcal L(\widehat M)\)-definable parameter set \(C\) and an \(\mathcal L(\varnothing)\)-formula \(\psi(y,z)\) whose fibers are exactly the Hausdorff limits of fibers from \(A\). Hence the Hausdorff limits form an \(\mathcal L(\widehat M)\)-definable family in \(\widehat{\mathcal M}\).
\end{proof}


\bibliographystyle{alpha}
\bibliography{references}

\end{document}